\documentclass[11pt]{amsart}
\usepackage{graphicx}
\usepackage[active]{srcltx}
 \makeatletter
\renewcommand*\subjclass[2][2000]{%
  \def\@subjclass{#2}%
  \@ifundefined{subjclassname@#1}{%
    \ClassWarning{\@classname}{Unknown edition (#1) of Mathematics
      Subject Classification; using '1991'.}%
  }{%
    \@xp\let\@xp\subjclassname\csname subjclassname@#1\endcsname
  }%
}
 \makeatother
\usepackage{enumerate,url,amssymb,  mathrsfs,yhmath}

\newtheorem{theorem}{Theorem}[section]
\newtheorem{lemma}[theorem]{Lemma}
\newtheorem*{lemma*}{Lemma}

\newtheorem{corollary}[theorem]{Corollary}

\def\1ton{1,2,\ldots,n}
\def\det{{\rm det}}

\def\Div{{\rm Div}}

\def\d{\textnormal d}

\usepackage{amssymb}
\usepackage{amsthm}
\usepackage{mathrsfs, amsfonts, amsmath}
\usepackage{graphicx}
\usepackage{psfrag}
\newcommand{\norm}{\,|\!|\,}
\newcommand{\bydef}{\stackrel{{\rm def}}{=\!\!=}}

\newcommand{\onto}{\xrightarrow[]{{}_{\!\!\textnormal{onto}\!\!}}}

\newcommand{\loc}{\text{loc}}
\newcommand{\Tr}{\text{Tr}}

\newcommand{\R}{\mathbb{R}}
\newcommand{\X}{\mathbb{X}}

\newcommand{\W}{\mathscr{W}}

\theoremstyle{definition}

\newtheorem{conjecture}[theorem]{Conjecture}

\theoremstyle{remark}
\newtheorem{remark}[theorem]{Remark}

\numberwithin{equation}{section}

\newcommand{\abs}[1]{\lvert#1\rvert}
\DeclareMathOperator{\Mod}{Mod}

\def\XXint#1#2#3{{\setbox0=\hbox{$#1{#2#3}{\int}$}
\vcenter{\hbox{$#2#3$}}\kern-.5\wd0}}

\def\ge{\geqslant}
\begin{document}

\title[$(n,\rho)-$harmonic mappings and energy minimal deformations ]{$(n,\rho)-$harmonic mappings and energy minimal deformations between annuli} \subjclass{Primary 31A05;
Secondary 42B30 }


\author{David Kalaj}
\address{University of Montenegro, Faculty of Natural Sciences and
Mathematics, Cetinjski put b.b. 81000 Podgorica, Montenegro}
\email{davidk@ac.me}

\begin{abstract}
We extend the main results obtained by Iwaniec and Onninen in Memoirs of the AMS (2012). In the paper it is solved the minimization problem of $(\rho,n)$ energy of Sobolev homeomorphisms between two concentric annuli in the Euclidean space $\mathbf{R}^n$. Here $\rho$ is a radial metric defined in the image annulus. The key of the proofs comes from the solution to the Euler-Lagrange equation for radial harmonic mapping. This is a new contribution  on the topic of famous Nitsche conjecture.
\end{abstract}  \maketitle


\section{Introduction}
Let $0<r<R$, $0<r_\ast<R_\ast$ and let  $\mathbb{A}=A(r,R)\bydef\{x:r<|x|<R\}$ and $\mathbb{A}_*=A(r_*, R_*)\bydef\{x:r_\ast<|x|<R_\ast\}$ be two annuli in the Euclidean space $\mathbb{R}^n$ equipped with the Euclidean norm $|\cdot|$. Let $\rho$  be a continuous function on the closure of $\mathbb{A}_*$.
The $(\rho,n)-$ energy integral of a mapping $h\in \W^{1,n}(\mathbb{A}, \mathbb{A}_*)$  is defined by \begin{equation}\label{penergy}\mathscr{E}_\rho[h]=\int_{A(r,R)} \rho(h(x)) \|Dh(x)\|^n dx.\end{equation} The central aim of this paper is to minimize the $(\rho,n)-$ energy integral between $\mathbb{A}$ and $\mathbb{A}_\ast$ throughout the class of homomorphisms from the Sobolev class $\W^{1,n}(\mathbb{A}, \mathbb{A}_*)$.  We will assume that $\rho$ is $C^1$ radial metric that is $\rho(y)=\rho(|y|)$ for $y\in \mathbb{A}_*$. We will also assume that
\begin{equation}\label{regularm}\min_{r_*\le s\le R_*} \rho(s)s^n=\rho(r_*)r_*^n\end{equation} and refer to those metrics as \emph{regular metrics}.  The modulus of $A(r,R)$ is defined by the formula $\Mod A(r,R)=\omega_{n-1}\log \frac{R}{r}$, where $\omega_{n-1}$ is the area of the unit sphere $\mathbf{S}^{n-1}$.

For a homomorphism $f$ of Sobolev class class $\W^{1,1}$ we say that has finite outer distortion if
$$\|D f\|^n\le n^{n/2}K_O[f,x]J_f(x),$$ where $1\le K_O[f,x]$ is measurable and the least function with the above property. Then $K_O[f,x]$ is called the outer distortion of $f$. Here $$\|D f\|=\sqrt{\left<Df, Df\right>}=\sqrt{\sum_{k=1}^n |Df(x)e_i|^2},$$ and $J_f$ is the determinant of the Jacobian matrix. A concept somehow dual is the inner distortion defined by the so-called co-factor matrix of $Df(x)$. Namely we define $\mathbf{adj}(Df(x))=J_f ( D^*f(x))^{-1}.$
Then $$\mathbb{K}_I[f,x] = \frac{\|\mathbf{adj}(Df(x))\|^{n}}{n^{n/2}\det(\mathbf{adj}(Df(x)))}=\frac{\|\mathbf{adj}(Df(x))\|^{n}}{n^{n/2}J_f^{n-1}(x)},$$ for $J_f(x)\neq 0$ and $K_I[f,x]=1$ for $J_f(x)=0$.

An important fact to be noticed in this introduction is the fact, if $\rho$ is a continuous function on the closure of $\mathbb{A}_*$ and  if $f\in \W^{1,n-1}(\mathbb{A}_*, \mathbb{A})$ is a homeomorphism, then its inverse mapping $h$ belongs to the Sobolev class $\W^{1,n}(\mathbb{A}, \mathbb{A}_*)$  and we have the following formula

\begin{equation}\label{disener}\int_{\mathbb{A}}\rho(h(x))\|Dh(x)\|^n dx =n^{n/2}\int_{\mathbb{A}_*}\rho(y)\mathbb{K}_I[f,y] dy .\end{equation}
Concerning the criteria of integrability of inverse mapping and related problems we refer to the papers \cite{mali, heko}.

In this paper we extend the main result and simplify the proofs in \cite{memoirs}. We made a unified approach to the minimizing problem of $(\rho,n)$-energy for the class of all $C^1$  radial metrics $\rho$ satisfying the condition $\rho(s)s^n$ is non-decreasing. This condition has been fulfilled by two metrics $\rho(s)\equiv 1$ and $\rho(s)\equiv s^{-n}$ considered by Iwaniec and Onninen in \cite{memoirs}. The paper generalizes  also the main results by Astala, Iwaniec and Martin in \cite{astala}, and also by the author in \cite{klondon} where it is treated the similar problem but only for the case $n=2$. The case of non rounded annuli and non radial metrics has been treated in the papers \cite{invent} and \cite{calculus} respectively, but also for the case $n=2$.  In this paper we assume that $n\ge 3$. The paper  is a continuation of study of the so-called Nitsche phenomenon, invented by J. C. C. Nitsche in \cite{Nitsche} who stated his famous conjecture. Further the conjecture has been proved by Iwaniec, Kovalev and Onninen in \cite{nconj}, after some partial results obtained by Weitsman \cite{W}, Lyzzaik \cite{L} and Kalaj \cite{Ka}. For its counterpart to general annuli on Riemannian surfaces we refer to the recent paper \cite{Ka1}. The Nitsche conjecture in the context of $(\rho,n)$-harmonic mappings is given in \eqref{nitb}. We prove in the first result (Theorem~\ref{newkalajrad}) that we can find a radial $(\rho,n)$-harmonic harmonic mapping between two annuli $\mathbb{A}$ and $\mathbb{A}_\ast$ if and only if the generalized Nitsche bound \eqref{nitb},  is satisfied. This bound said roughly speaking that if we have a $(\rho,n)$ harmonic diffeomorphism between annuli $\mathbb{A}$ and $\mathbb{A}_\ast$, then the image annulus cannot bee too thin, but can be arbitrary thick. On the other hand this bound is equivalent with the fact that the $(\rho,n)-$ energy integral is minimized for a certain radial $(\rho,n)-$harmonic diffeomorphism if $n=3$; if $n\ge 4$, then we have some obstruction, and in this case the image annulus cannot be too thick (Theorem~\ref{newkalaj}), in order that the radial mapping is a minimizer. The precise estimate how thick the image annulus could be, remains an open interesting problem.

\subsection{$(\rho,n)$-harmonic equation}
The classical Dirichlet problem concerns the energy minimal mapping $h \colon \mathbb{A} \to \mathbb{R}^n$ of the Sobolev class $h\in h_\circ + \W^{1,n}_\circ (\mathbb{A}, \mathbb{R}^n)$ whose boundary values are explicitly prescribed by means of a given mapping $h_\circ \in  \W^{1,n} (\mathbb{A}, \mathbb{R}^n)$.  Let us consider the variation   $h \leadsto h\,+ \,\epsilon \eta $,  in which $\eta \in \mathscr C^\infty_\circ (\X , \R^n)$ and $\epsilon \to 0$, leads to the integral form of the familiar $n$-harmonic system of equations
\begin{equation}\label{equa1}
\int_{\mathbb{A}} \langle \rho(h) \norm Dh\norm^{n-2}Dh , \, D\eta \rangle =0, \quad \mbox{ for every } \eta \in \mathscr C^\infty_\circ (\X , \R^n).
\end{equation}
Equivalently
\begin{equation}\label{equa2}
\Delta_n h = \Div (\big( \rho(h)\norm Dh \norm^{n-2}Dh\big)=0, \quad \mbox{in the  sense of distributions.}
\end{equation}

Similarly as in in \cite{memoirs}, it can be derived the general $(\rho,n)$-harmonic equation which by using a different variation as the following.

The situation is different if  we allow $h$ to slip freely along the boundaries. The {\it inner variation} come to stage in this case. This is simply a change of the  variable; $h_\epsilon=h \circ \eta_\epsilon $, where  $\eta_\epsilon \colon \X \onto \X$ is a $\mathscr C^\infty$-smooth diffeomorphsm  of $\X$ onto itself, depending smoothly on a parameter $\epsilon \approx 0$ where  $\eta_\circ = id \colon \X \onto \X$. Let us take on the inner variation of the form
\begin{equation}\label{equa7}
\eta_\epsilon (x)= x + \epsilon \, \eta (x), \qquad \eta \in \mathscr C_\circ^\infty (\X, \R^n).
\end{equation}
By using the notation  $y=x+\epsilon \, \eta (x) \in \X$,
we obtain
$$\rho(h_\epsilon)Dh_\epsilon (x) = \rho(h(y)) Dh (y) (I+ \epsilon D\eta).$$ Hence
\[
\begin{split}
\rho(h(y))\norm Dh (y)\norm^n & = \rho(h(y))\norm Dh(y)\norm^n
\\&+ n \epsilon\,  \rho(h(y))\langle \norm{Dh}\norm ^{n-2} D^\ast h(y)\cdot  Dh(y)\, ,\,  D \eta \rangle + o(\epsilon).
\end{split}
\]
Integration with respect to $x\in \X$ we obtain
\[\mathscr E_\rho[h_\epsilon] = \int_\X \left[ \rho(h_\epsilon)\norm Dh \norm^n +  n \epsilon \rho(h_\epsilon)\langle \norm{Dh}\norm ^{n-2} D^\ast h\cdot  Dh\, ,\,  D \eta \rangle \right]\, \d x + o(\epsilon) .\]
We now make the substitution $y=x + \epsilon \, \eta (x)$, which is a diffeomorphism for small $\epsilon$, for which we have:  $x= y- \epsilon \, \eta (y)+ o(\epsilon)$, $\eta (x)= \eta (y)+o(1)$, $\rho(h(x))=\rho(\rho(y))+o(1)$ and the change of volume element $\d x = [1-\epsilon \, \Tr \,D \eta (y) ]\, \d y + o(\epsilon) $. Further
$$\int_\X \rho(h(y))\norm Dh(y) \norm^n \d x=\int_\X \rho(h(y))\norm Dh(y) \norm^n [1-\epsilon \, \Tr \,D \eta (y) ]\, \d y + o(\epsilon)$$
 The so called equilibrium equation for the inner variation is obtained from $\frac{\d}{\d \epsilon} \mathscr E_{h_\epsilon}\,=\,0\,$ at $\epsilon =0$,
\begin{equation}\label{intstar}\int_\X \langle \rho(h)\norm Dh \norm^{n-2} D^\ast h \cdot Dh - \frac{\rho(h)}{n} \norm Dh \norm^n I \, , \, D \eta \rangle \, \d y=0 \end{equation}
or, by using distributions
\begin{equation}\label{enhe}
\Div \left(\rho(h)\norm Dh \norm^{n-2} D^\ast h \cdot Dh - \frac{\rho(h)}{n} \norm Dh \norm^n I  \right)=0.
\end{equation}

The name {\it generalized $n$-harmonic equation} is given to ~\eqref{enhe} because of the following:
\begin{lemma}\label{leiwa}
Every $n$-harmonic mapping $h\in \W^{1,n}_{\loc} (\X , \R^n)$ solves the generalized $n$-harmonic equation~\eqref{enhe}.
\end{lemma}
\begin{proof}
It is the similar as in \cite{memoirs} because $\rho$ make no essential changes for the proof.
\end{proof}
In dimension $n=2$, the generalized harmonic equation reduces to
\begin{equation}\label{e2he}
\Div \left(\rho(h)D^\ast h \, Dh - \frac{\rho(h)}{2} \norm Dh \norm^2 I  \right)=0.
\end{equation}
This equation is known as Hopf equation, and the corresponding differential is called the Hopf differential. Since for $h(z)=(a(z),b(z))$, we have $$\rho(h)D^\ast h \, Dh - \frac{\rho(h)}{2} \norm Dh \norm^2 I  =\left(
                                      \begin{array}{cc}
                                        U & V \\
                                        V & -U \\
                                      \end{array}\right),$$ where $$U = \frac{\rho(h)}{2} (a_x^2+b_x^2-a_y^2-b_y^2)$$ and $$V=\rho(h)(a_xa_y+b_xb_y),$$                                 then \eqref{e2he} in complex notation  takes the form $$(U_x +U_y)- i (V_x+V_y)=0$$ or what is the same
\begin{equation}\label{eqv}
\frac{\partial}{\partial \bar z} \left(\rho(h(z))h_z \overline{h_{\bar z}}\right)=0, \qquad z= x+iy.
\end{equation}
In \cite{invent} and \cite{calculus}, it is used the fact that Hopf's differential of a minimizer has special form namely  $$\rho(h(z))h_z \overline{h_{\bar z}}=\frac{c}{z^2}$$ for a certain constant $c$ that  depends on the ration of modulus of annuli. In this paper, this constant $c$ will be also crucial for proving the minimization result.

If, in addition $h\in \mathscr C^2$ then~\eqref{eqv} is equivalent with
\begin{equation}\label{eqvi}h_{z\overline z}+{(\log \rho)}_w\circ h\cdot  h_z\,h_{\bar z}=0,
\end{equation}
which is known as the harmonic mapping equation. In particular, if $\rho(w)=(1-|w|^2)^{-2}$, then the equation produces hyperbolic harmonic mappings. The class is specially interesting, due to recent discover that every quasisimmetric map of the unit circle onto itself can be extended to a quasiconformal hyperbolic harmonic mapping of the unit disk onto itself. This problem  is known as the Schoen conjecture and it was solved recently in positive by Markovi\'c in \cite{markovic}.
\section{Radial solutions to the generalized $n$-harmonic equation}
We assume that $R>1,$ and $R_\ast>1$ and $\mathbb{A}=A(1,R)$, $\mathbb{A}_\ast=A(1,R_\ast)$.
Recall that $\rho$ is a radial $C^1$ function in $\mathbb{A}_\ast=A(1,R_\ast)$ so that $\rho(s)s^n$ attains its minimum for $s=1$.
Let us consider a radial mapping
\[h(x)= H\big(\abs{x}\big) \frac{x}{\abs{x}}, \qquad \mbox{where $H=H(t)$ is $C^2$}\]
We find that
\begin{equation}\label{eqvii}
\begin{split}
\Lambda & = \rho(h)\norm Dh \norm^{n-2} \left(D^\ast h \cdot Dh - \frac{1}{n} \norm Dh \norm^2 I\right)\\
& = \rho(h) (n-1)^\frac{n-2}{n} \left(H^2 + \frac{\abs{x}^2 \dot{H}^2}{n-1}  \right)^\frac{n-2}{2} \left(H^2 - \abs{x}^2 \dot{H}^2\right) \frac{1}{\abs{x}^n} \left(\frac{x \otimes x}{\abs{x}^2}- \frac{1}{n}I\right).
\end{split}
\end{equation}
Thus \eqref{e2he} reduces to \begin{equation}\label{redu}
\Div \Lambda\equiv 0.
\end{equation}
We show  that if $h$ is a $\mathscr C^2$-smooth $n$-harmonic mapping then $H=H(t)$ must satisfy the  {\it characteristic equation}
\begin{equation}\label{eqviii}
 \rho(h)\left(H^2 + \frac{\abs{x}^2 \dot{H}^2}{n-1}  \right)^\frac{n-2}{2}\cdot\left(H^2 - |x|^2 \dot{H}^2\right)\equiv \mbox{const.}
\end{equation}

Assume that $h=H(t) \frac{x}{|x|}$, $t=|x|$, is a radial function, where $H$ is real diffeomorphism between intervals $[1,R]$ and $[1,R_*]$, then by a direct calculation we obtain $$\|Dh\|^2 = \dot H^2(t)+(n-1) \frac{H^2(t)}{t^2}, $$ and $$J_h=\frac{\dot HH^{n-1}}{t^{n-1}}.$$

Thus

$$\mathscr{E}_\rho[h]=\mathcal{E}[H]\bydef\omega_{n-1} \int_1^R\rho(H(t)) t^{n-1} \dot H^2(t)+(n-1) H^2(t)/t^2)^{n/2} dt.$$
If $$ L[t,H,\dot H]\bydef\rho(H(t)) t^{n-1} \left( \dot H(t)^2+(n-1) \frac{H^2(t)}{t^2}\right)^{n/2}$$ then the Euler-Lagrange equation is

\begin{equation}\label{elag}L_H = \frac{\partial}{\partial t} L_{\dot H} .\end{equation} Then \eqref{elag} is equivalent with \eqref{redu}, because $H\in C^2$. Further \eqref{elag} reduces to

\[\begin{split} M\bydef &\frac{1}{(n-1) H(t)^2+t^2 \dot H^2}(n-1) t^{n-1} \left(\frac{(n-1) H(t)^2}{t^2}+\dot H^2\right)^{n/2}
 \\ &\times \bigg(-n \rho[H(t)] \left(H(t)-t \dot H\right) \left((n-1) H(t)^2+(-2+n) t H(t) \dot H+t^2 \dot H^2\right)\\&-\left(H(t)^2-t^2 \dot H^2\right) \left((n-1) H(t)^2+t^2 \dot H^2\right) \rho'[H(t)]\\&+n t^2 \rho[H(t)] \left(H(t)^2+t^2 \dot H^2\right) \ddot H\bigg)=0.\end{split}\]

Now we have the following key formula for our approach
\[\begin{split}M&=\frac{(n-1) t^{-1-n} \left((n-1) H(t)^2+t^2 \dot H^2\right)}{\dot H}\\ &\times \frac{\partial}{{\partial t}} \left(\rho[H(t)]\left(H(t)^2-t^2 \dot H^2\right) \left((n-1) H(t)^2+t^2 \dot H^2\right)^{\frac{1}{2} (n-2)}\right)=0.\end{split}\]

Thus we obtain

\begin{equation}\label{banane}\mathcal{L}[H]=\rho[H(t)]\left(H(t)^2-t^2 \dot H^2\right) \left( H(t)^2+\frac{t^2 \dot H^2}{n-1}\right)^{\frac{1}{2} (n-2)}\equiv c. \end{equation}

Further we look at increasing diffeomorphisms $H$ between two intervals $[1,R]$ and  $[1,R_\ast]$ that are solutions of the previous equation. Then $$c=\rho[H(t)]\left(H(t)^2-t^2 \dot H^2\right) \left( H(t)^2+\frac{t^2 \dot H^2}{n-1}\right)^{\frac{1}{2} (n-2)}$$
Since the function $$\psi(b)\bydef \left(a^2-b^2\right) \left(a^2+\frac{b^2}{n-1}\right)^{\frac{1}{2} (n-2)}$$ is decreasing, because
$$\psi'(b)=-\frac{b \left(a^2+b^2\right) \left(a^2+\frac{b^2}{n-1}\right)^{n/2} (n-1) n}{\left(b^2+a^2 (n-1)\right)^2}$$ we obtain that
$$c\le \rho[H(t)]H(t)^n,$$ and thus $$c\le \min\{\rho[H(t)]H(t)^n,1\le |t|\le R\}=\rho[H(1)]H(1)^n=\rho(1).$$
Thus we conclude that if the equation has a solution then \begin{equation}\label{nitb}c\le c_\diamond\bydef \rho(1)\ \ \ \text{\bf (Nitsche inequality)}.\end{equation}

Let us demonstrate the connection of \eqref{nitb} with the standard Nitsche inequality.

In this special case $\rho\equiv 1$ and $n=2$. So the inequality \eqref{nitb} is equivalent with the inequality

$$H(t)^2-t^2 \dot H^2\le 1.$$ Assuming that $\dot H\ge 0$, $H(1)=1$ and $H(R)=R_\ast$  then the last inequality is equivalent with $$\int_1^R\frac{dr}{r} \le \int_1^{R_\ast} \frac{dH}{\sqrt{H^2-1}}$$ or what is the same $$ \log\left[R_\ast+\sqrt{R^2_\ast-1}\right]\ge \log R.$$ Thus we obtain \begin{equation}\label{stand}
R_\ast \ge \frac{1+R^2}{2R},\end{equation} and this is the standard Nitsche inequality. Recall that the condition \eqref{stand} is sufficient and necessary that there exists a planar harmonic diffeomorphism between annuli $A(1,R)$ and $A(1, R_\ast)$ (\cite{nconj}).

We will prove that the condition \eqref{nitb} is equivalent with the fact that there  exists  a radial $(\rho,n)-$harmonic diffeomorphism between given annuli and conjecture that
\begin{conjecture}
There is a $(\rho,n)$ harmonic mappings between annuli $A(1,R)$ and $A(1,R_\ast)$ if and only \eqref{nitb} holds.
\end{conjecture}

The conjecture will be verified on the class of minimizers of $(\rho,n)$  energy (Theorem~\ref{newkalaj}).

 Let  $$\eta_H(t)=\eta(t)\bydef \frac{t^2(\dot H(t))^2}{H^2(t)},$$ and let $$\zeta(t)\bydef \eta^2(t).$$ Assume also that the constant  $c$ satisfies \eqref{nitb}. Then the equation $\mathcal{L}[H]=c$ is equivalent with the equation
$$(1-\eta^2(t))\left(1+\frac{\eta^2(t)}{n-1}\right)^{(n-2)/2}=v_c(H(t)),$$ or
$$ \Phi(\zeta)\bydef(1-\zeta(t))\left(1+\frac{\zeta(t)}{n-1}\right)^{(n-2)/2}=v_c(H(t)),$$
where $$v_c(H)\bydef\frac{c}{H^{n}\rho(H)}\le 1.$$
Since $$\Phi'(\zeta)=-\frac{1}{2} (n-1)^{1-\frac{n}{2}} n (1+\zeta) (n-1+\zeta)^{-2+\frac{n}{2}},$$
we conclude that $\Phi:[0,\infty)\to (-\infty ,1]$ is strictly decreasing  and smooth function and thus a diffeomorphism. Moreover $\Phi(0)=1$, $\Phi(1)=0$ and
$\Phi(\infty)=-\infty$. Let $\Psi=\Phi^{-1}:(-\infty ,1]\to [0,\infty) $. Then $\Psi$ is strictly decreasing as well with $\Psi(0)=1$ and thus
\begin{equation}\label{negative}
\Psi(\zeta)\ge 1, \ \ \ \text{if}\ \ \ \zeta\le 0.
\end{equation}
 Then $$\zeta(t)=\Psi(v_c(t)).$$

Further \begin{equation}\label{eta1}\eta_H(t)=\frac{tH'(t)}{H(t)}=\sqrt{\Psi(v_c(H(t)))}.\end{equation} Now by taking the initial condition $H(1)=1$, we arrive to the implicit solution  $$\log t = \int_1^s \frac{1}{y\sqrt{\Psi(v_c(y))}} \d y,$$ with $s=H(t)$.

Thus for $$T_c(s)\bydef\exp\left[\int_1^s \frac{1}{y\sqrt{\Psi(v_c(y))}} \d y\right],$$ the diffeomorphism  \begin{equation}\label{HcT}H_c\bydef T_c^{-1}\end{equation} is a solution of the equation $\mathcal{L}[H](t)=c$ with the initial condition $H(1)=1$ and $H'(1)\ge 0$. Further $T_c(R_\ast)=R$, where
\begin{equation}\label{R}R= \exp\left[\int_1^{R_\ast} \frac{1}{y\sqrt{\Psi(v_c(y))}} \d y\right].\end{equation}
Let us emphasis the following important fact. Every parameter from the set $\{R, R_\ast, c\}$ is uniquely determined by two others. More precisely, we have $$ c=c(R, R_\ast), \ \ R=R(c,R_\ast), \ \ \text{and}\ \  R_\ast=R_\ast(c, R).$$

Since $H_c(t)$ is increasing, then  $v_c(H_c(t))$ is decreasing for $c>0$ and increasing for $c<0$, and thus   $\sqrt{\Psi(v_c(H(t)))}$ increases for $c\ge 0$ and decreases for $c\le 0$. Thus we obtain that  \begin{equation}\label{eta}
\left\{
           \begin{array}{ll}
             \eta_H(t)\ge 1, & \hbox{ and $\eta_H(t)$ increases on $[1,R]$ if $c<0$;} \\
             \eta_H(t)\le 1, & \hbox{and  $\eta_H(t)$ decreases on $[1,R]$ if $c>  0$.}\\
              \eta_H(t)\equiv  1, & \hbox{on $[1,R]$ if $c=  0$.}
           \end{array}
         \right.
\end{equation}
Let $n\ge 4$ and let $\kappa_n$ be the solution of the equation $$(n-1+\eta^2) ^{\frac{n-2}{2}}(\eta^2-1)=\eta^n$$ on the interval $[1,\frac{\sqrt{n-1}}{\sqrt{n-3}}]$. Then for $n\ge 4$ we put  $c_\diamond:=-{\rho(1)}\kappa_n^n$. If $n=3$ we put $c^\diamond:=-\infty$.

Furthermore, if $c^\diamond\le c_1<c_2\le c_\diamond$, then $T_{c_1}(R_\ast)<T_{c_2}(R_\ast)$. Now if $H_{c_1}(R)=R_\ast$, then we have $R<T_{c_2}(H_{c_1}(R))$ and thus $H_{c_2}(R)\le H_{c_1}(R).$ If we use the convention $H_{c^\diamond}\equiv +\infty$ for $n=3$, then we infer that for every $n\ge 3$
\begin{equation}\label{cdiam}
H_{c_\diamond}(R)<H_{c^\diamond}(R), \ \ \ R>1.
\end{equation}
We conclude this section by proving the following theorem.
\begin{theorem}\label{newkalajrad}
Let $R>1$ be fixed. Let $\rho$ be a regular metric on $\mathbb{A}=A(1,R)$.   If $R_\ast>  1$, then there is a radial $(\rho,n)-$harmonic diffeomorphism $h=h_c$ between annuli $\mathbb{A}=A(1, R)$ and $\mathbb{A}^\ast=A(1, R_\ast)$ if and only if
\begin{equation}\label{Kalajnitsche} H_{c_\diamond}(R)\le R_\ast,\end{equation} or equivalently if
\begin{equation}\label{nitsche} c(R,R_\ast)\le \rho(1).\end{equation}
\end{theorem}
\begin{proof}
 Let $R_\ast\ge H_{ c_\diamond }(R)$. We prove that there is $c\le  c_\diamond $ so that $H_{c}(R)=R_\ast$. Then \begin{equation}\label{hhc}h_c(x)\bydef H_c(|x|)\frac{x}{|x|}\end{equation} is a $n-$harmonic diffeomorphism between $A(1, R)$ and $A(1, R_\ast)$. In order to do so define the function $$\Lambda(c)=\exp\left[\int_1^{R_\ast} \frac{1}{h\sqrt{\Psi(v_c(h))}} dh\right].$$

Then $\Lambda(c)$ is continuous for $c\in(-\infty, c_\diamond ]$. Moreover the function $c\to \Psi(v_c(y))$ is increasing for fixed $y$ and so $c\to y\sqrt{\Psi(v_c(y))}$ is strictly decreasing. So $c\to \frac{1}{y\sqrt{\Psi(v_c(y))}}$ is increasing and thus $\Lambda$ is increasing. As $\Lambda(-\infty)=1$, by Mean value theorem there is a unique $c$ so that $\Lambda(c)=R_\ast$.

To prove the opposite part, assume that $h=H(|x|)\frac{x}{|x|}$ is a harmonic diffeomorphism between annuli $A(1,R)$ and $A(1,R_\ast)$. Then, because of Lemma~\ref{leiwa}, for a constant $c$,
$\mathcal{L}[H]=c$. Further, $H$ is a diffeomorphism, and so $c\le  c_\diamond $. It follows that \eqref{Kalajnitsche}.
\end{proof}
\section{The main result}
\begin{theorem}\label{newkalaj} Assume that  $\rho$ is a regular metric in $\mathbb{A}_\ast=A(1,R_\ast)$, $R_\ast>1$.

a) Let $R>1$ be fixed. We have the sharp inequality \begin{equation}\label{sharpineq}\int_{A(1,R)}\rho(|h|) \|Dh\|^n \ge  \int_{A(1,R)}\rho(|h_c|) \|Dh_c\|^n ,\end{equation} for orientation preserving homeomorphisms of class $\W^{1,n}$ between $A(1,R)$ and $A(1, R_\ast)$ mapping the inner boundary onto the inner boundary   if
 \begin{equation}\label{Kalajnitsche1} H_{c_\diamond}(R)\le R_\ast\le H_{c^\diamond}(R),\end{equation} or in its equivalent form if
  \begin{equation}\label{nitsche1} -{\rho(1)}\kappa_n^n\le  c\le \rho(1).\end{equation}
  The equality is attained if and only if $h=\mathcal{T} h_c$ where $\mathcal{T}$ is a linear isometry of $\mathbf{R}^n$.

 b) If $R_*>1$ and if \begin{equation}\label{cr}c<-\frac{2\rho(R_*)R_*^n}{{n-2}} \left(\frac{n-2}{n-3}\right)^{n/2},\end{equation} then the $(\rho,n)$-harmonic diffeomorphism $h_c=H_c(|x|)\frac{x}{|x|}:A(1,R)\to A(1,R_*)$ is not the minimizer of the functional of energy. Here $R=R(c,R_*)$ is defined in \eqref{R}.
\end{theorem}
By using \eqref{disener} and Theorem~\ref{newkalaj} we obtain

\begin{corollary} Assume that $\mathbb{A}_\ast=A(1,R_\ast)$ is an annulus and  assume that $\rho$ is a regular metric on $\mathbb{A}_\ast$. For  $-{\rho(1)}\kappa_n^n\le  c\le \rho(1)$ let  $R=R(c,R_\ast)$ and let $f^c=(h_c)^{-1}$. Then we have the following sharp inequality

\begin{equation}\label{sharpineq1}\int_{\mathbb{A}_\ast}\rho(y) \mathbb{K}_I[f,y] \ge  \int_{\mathbb{A}_\ast}\rho(y) \mathbb{K}_I[f^c,y],\end{equation} for every homeomorphism $f:\mathbb{A}_\ast\to \mathbb{A}$ preserving the inner boundary and the orientation and belonging to the Sobolev space $\W^{1,n-1}$.

\end{corollary}
\begin{remark}

The question arises how general can be two double connected domains, in order to obtain similar result.
\begin{itemize}
  \item

 Instead of $A(1, R)$ and $A(1, R_\ast)$, we could take the annuli $A(r, R)$ and $A(r_\ast
, R_\ast)$. The last case reduces to the previous one because the $(\rho,n)$-harmonic mappings are invariant under homothety of domain and of image domain. Namely if $h$ is $(\rho,n)$ harmonic mapping between annuli $A(r, R)$ and $A(r_\ast, R_\ast)$, then $\lambda h(\mu x)$ is harmonic as well w.r.t. the metric $\rho(\mu|x|)$ between annuli $A(\mu r, \mu R)$ and $A(\lambda r_\ast, \lambda R_\ast)$.

\item If $h:A(1,R)\to A(1, R_\ast)$ is a harmonic homeomorphism that map the inner boundary onto the outer  boundary, then $$h_1(x)=h\left (R\frac{x}{|x|^2}\right):A(1,R)\to A(1, R_\ast)$$ that map the inner boundary onto the inner boundary. This follows from the fact that the class of $n-$harmonic mappings is invariant under postcomposing by conformal mappings of the space, exactly as in the planar case. More precisely, if $h:D\to \mathbf{R}^n$ is $(\rho,n)-$harmonic, then $h\circ T$ is $(\rho,n)-$ harmonic in $D'=T^{-1}(D)$, for every metric $\rho$ and every M\"obius transformation $T$ on the space $\mathbf{R}^n$. Here $D\subset \mathbf{R}^n$ is an open subset. This follows from the following formulas
\[\begin{split}\mathscr{E}_\rho[h]&=\int_D \rho(h(x)) \|Dh(x)\|^n dx\\&= \int_D \rho(h(x)) \|Dh(x)\|^n dx\\&=\int_{D'} \rho(h(T(y))) \|Dh(T(y))\|^n J_T(y)dy
\\&= \int_{D'} \rho(h\circ T(y)) \|Dh(T(y))\|^n |DT(y)|^n dy\\&=\int_{D'} \rho(h\circ T(y)) \|D(h\circ T)(y)\|^n  dy=\mathscr{E}_\rho[h\circ T]\end{split}\]
Here $|DT(y)|\bydef \max_{|h|=1}|DT(y) h|$. Thus if $h$ is a stationary point of energy integral, then so is $h\circ T$.

 \item If $f$ is a $n-$harmonic mapping between annuli $A(r,R)$ and $A(r_\ast, R_\ast)$ w.r.t. the metric $\rho(|w|)$, then $\tilde  f(x)=\frac{f(x)}{|f(x)|^2}$ is a $n-$harmonic mapping between annuli $A(r,R)$ and $A(1/R_\ast, 1/r_\ast)$ w.r.t. the metric $$\tilde \rho(|\omega|)=|\omega|^n{\rho\left(\frac{1}{|\omega|}\right)}.$$
Namely, if $g(x)=\frac{x}{|x|^2}$, then $g$ is conformal and thus $$\left<Dg(x)h, Dg(x)k\right> =|Dg(x)|^2\left<h,k\right>,$$ where $$|Dg(x)|=\max_{|h|=1}|Dg(x)h|.$$
 Here $$Dg(x)h = \frac{h}{|x|^2}- \frac{x\left<x,h\right>}{|x|^4},$$ and thus $$\|Dg(x)\|^2=n|Dg(x)|^2=\frac{n}{|x|^2}.$$ Further we obtain
 \[\begin{split}\|D\tilde f\|^2&= \mathrm{Tr}((D\tilde f)^{*}D\tilde f )=\sum_{k=1}^n \left<D\tilde f e_k, D\tilde f e_k\right>
 \\&=\sum_{k=1}^n \left<g'(f(x))D f e_k, g'(f(x))D f e_k\right>
 =\frac{\|D f\|^2}{|f|^2},\end{split}\]and so $$\mathcal{E}_{\tilde \rho}[ \tilde f]=\mathcal{E}_{\rho}[f].$$ So if $f$ is the minimizer of $\mathcal{E}_{\rho}$ then $\tilde f$ is the minimizer of $\mathcal{E}_{\tilde \rho}$.
 \item  The main result can be formulated for a little more general case, namely for two double connected domains whose boundry components are two spheres (which are not concentric). Namely for annuli $\mathcal{A}=T_1(\mathbb{A})$ and $\mathcal{A}_\ast=T_1(\mathbb{A}_\ast)$, where $T_1$ and $T_2$ are certain   M\"obius transformations of the space $\mathbb{R}^n$. The class of conformal mappings on the space is very rigid, indeed it coincides with the class of M\"obius transformations. The planar case is far more interesting but also more difficult in this context (cf. \cite{invent,calculus}).
     \end{itemize}
\end{remark}
\begin{remark}
If  we take the substitution  $$\eta(t)=\frac{tH'(t)}{H(t)}$$ in \eqref{banane} we obtain  $$n+\frac{n t \left(1+\eta[t]^2\right) \eta'[t]}{\left(-1+\eta[t]^2\right) \left(n-1+\eta[t]^2\right)}+\frac{H(t) R'[H(t)]}{R[H(t)]}=0.$$
In particular if $\rho(s)=s^\nu$, we have \begin{equation}\label{etazeta}\frac{ \left(1+\eta[t]^2\right) \eta'[t]}{\left(1-\eta[t]^2\right) \left(n-1+\eta[t]^2\right)}=\frac{n+\nu}{nt}.\end{equation} By following the approach as in \cite{memoirs}, where are considered the special cases $\nu=0$ and $\nu=-n$, we can find that the solution $H=H_c$ can be expressed by mean of the so called elasticity function $\eta=\eta_H$, and has the similar features as in the case $\nu=0$ (see \cite[p.~35-42]{memoirs}). However we do not need those properties in order to prove our main result. Instead, we use only some general  results regarding the modulus of annuli obtained in \cite{memoirs} (Corollary~\ref{rrjedh}).
\end{remark}

For $x\in {\mathbb A}$ let $N=\frac{x}{|x|}$. Further let $T_2$, $\dots$, $T_n$ be $n-1$ unit vectors mutually orthogonal and orthogonal to $N$. Denote by $h_N$, $h_{T_2}$, $\dots$, $h_{T_{n}}$ the corresponding directional derivatives. Use the notation $$\left|h_T\right|=\sqrt{\frac{|h_{T_2}|^2+\dots +|h_{T_n}|^2}{n-1}}.$$ Then we have
\begin{corollary}\label{rrjedh}\cite{memoirs}.
Let $h$ be a homeomorphism between spherical rings ${\mathbb A}$ and
${\mathbb A}^\ast$ in the Sobolev class  $ {\mathscr
W}^{1,n}({\mathbb A}, {\mathbb A}^\ast)$. Then
\begin{equation}\label{E184}
 \int_{\mathbb A} \Phi \big( |h| \big ) \left|h_N\right|\,  \left|h_T\right|^{n-1} \ge \omega_{n-1} \int_{1}^{R_\ast} \tau^{n-1} \Phi(\tau)\, d\tau
\end{equation}
whenever  $\Phi$ is  integrable in $[1, R_\ast]$. We have the
equality in (\ref{E184}) if and only if \begin{equation}\label{jaceq}\left|h_N\right|\,
\left|h_T\right|^{n-1}=J_h(x).\end{equation} Furthermore,
\begin{equation}\label{E185}
\int_{\mathbb A} \frac{\left|h_N\right|}{|h|\, |x|^{n-1}} \ge
\textnormal{Mod}\, {\mathbb A}^\ast
\end{equation}
\begin{equation}\label{vienna12}
\int_{\mathbb A} \frac{ \left|h_T\right|^{n-1}}{|h|^{n-1}\, |x|} \ge   \textnormal{Mod}\, {\mathbb A}
\end{equation}
\end{corollary}
Note that we have equalities  if $h$ is a radial mapping.

We also need the following simple lemmas.
\begin{lemma}\label{le5}\cite{memoirs}
Let $u,v \ge 0$ and $0 \le  \sigma \le 1$. Then
\begin{equation}
\left[u^2 +(n-1)v^2\right]^\frac{n}{2} \ge a(\sigma)\,  v^n +b(\sigma)\, u
v^{n-1}
\end{equation}
where
\begin{equation}\label{aal}
a(\sigma)= (n-1) \left(\sigma^2 +n-1\right)^\frac{n-2}{2} \left(1-
\sigma^2\right)
\end{equation}
and
\begin{equation}\label{eqb}
b(\sigma)=  n \sigma \left(\sigma^2 +n-1 \right)^\frac{n-2}{2}
\end{equation}
Equality holds if and only if $u = \sigma v$.
\end{lemma}
\begin{lemma}\label{le7}\cite{memoirs}
Let $u,v \ge 0$ and $1 \le \sigma < \sigma_n$. Then
\begin{equation}\label{ES208}
a=a(\sigma)\bydef \frac{(\sigma^2+n-1)^\frac{n-2}{2}
(\sigma^2-1)}{\sigma^n}<1
\end{equation}
and, we have
\begin{equation}\label{E212}
\left[u^2 +(n-1)v^2\right]^\frac{n}{2} \ge a\,  u^n + b\,  u
v^{n-1}
\end{equation}
where
\begin{equation}\label{defb}
b= \frac{n\left(\sigma^2 + n-1\right)^\frac{n-2}{2}}{\sigma}
\end{equation}
Equality holds if and only if $u = \sigma v$.
\end{lemma}
\section[Proof of Theorem~\ref{newkalaj}]{Proof of Theorem \ref{newkalaj}}
$\clubsuit$ {\bf Proof of a)}.\\
Here we assume \eqref{Kalajnitsche1}. This bound means that there is a radial
$(\rho,n)$-harmonic homeomorphism
\begin{equation}\label{vienna3}
h_c : {\mathbb A} \rightarrow {\mathbb A}^{\! \ast}\, ,
\hskip1cm h_c (x) = H_c\big(|x|\big)\, \frac{x}{|x|}
\end{equation}
Recall the characterictic equation for $H=H_c(t)$ is
\begin{equation}\label{E954376}
\left[H^2 + \frac{t^2 \dot{H}^2}{n-1}  \right]^\frac{n-2}{2} \left(H^2 - t^2 \dot{H}^2\right) \equiv \frac{c}{\rho(H)}
\end{equation}
where $c=C(R,R_*)$ is a constant determined by $(R,R_*)$.
\\
{$\bullet $ \textbf{The case} $0\le c\le c^\diamond.$}\\
Let $$\eta_{_H}=\frac{t \dot{H}}{\dot{H}}.$$ Then \eqref{E954376} is equivalent with
\begin{equation}\label{E314}
\left(1 + \frac{\eta_{_H}^2}{n-1}\right)^\frac{n-2}{2} \left(1-
\eta_{_H}^2 \right) = \frac{c}{\rho(H)H^n}.
\end{equation}
Here $c$ satisfies the condition $c\le \rho(1)$ and so $$\frac{c}{\rho(H)H^n}\le 1.$$
Now, let $h : {\mathbb A}
\overset{\textnormal{\tiny{onto}}}{\longrightarrow} {\mathbb A}^{\!
\ast}$, $h \in {\mathscr
W}^{1,n}({\mathbb A}, {\mathbb A}^{\! \ast} )$, be arbitrary orientation preserving  homeomorphism of annuli mapping the inner boundary onto the inner boundary. For  $x\in
{\mathbb A}$,  let $u=
\left|h_{_N}(x)\right|$, $v= \left|h_{_T}(x)\right|$.
The equation (\ref{E314}) suggests that we should consider the nonnegative solution
$\eta=\eta(t)$  to the equation
\begin{equation}\label{equa}
\left(1+\frac{\eta^2}{n-1}\right)^\frac{n-2}{2} (1-\eta^2)=
\frac{c}{\rho(t)t^n} \, , \; \; 1 < t < R_\ast.
\end{equation}
There is exactly one such $\eta$ and it lies in the interval
$[0, 1]$ because $$\frac{c}{\rho(t)t^n}\le \frac{c}{\rho(1)}\le 1.$$
Let $\sigma=
\eta \big(t\big)$ be the solution of \eqref{equa}, where $t=|h(x)|$. Then $0\le \sigma \le 1$.  We apply Lemma \ref{le5} to obtain the point-wise
inequality
\begin{eqnarray}\label{davidi}
\rho(h)\norm Dh \norm^n &=& \rho(h)\left[\left|h_N\right|^2 +(n-1) \left|h_T\right|^2 \right]^\frac{n}{2} \nonumber \\
&\ge & \rho(h)a(\sigma)
\,\left|h_T\right|^n +\rho(h)b(\sigma)
\left|h_N\right|\, \left|h_T\right|^{n-1}
\end{eqnarray}
Now we find
\[\begin{split}
a(\sigma)&= (n-1) \left(\sigma^2 +n-1\right)^\frac{n-2}{2} \left(1-
\sigma^2\right)\\&=(n-1)^\frac{n}{2}\frac{c}{|h|^n \rho(|h|)}
\end{split}
\]
and so
\begin{eqnarray}\label{davidi1}
\rho(h)\norm Dh \norm^n &=& \rho(h)\left[\left|h_N\right|^2 +(n-1) \left|h_T\right|^2 \right]^\frac{n}{2} \nonumber \\
&\ge & (n-1)^\frac{n}{2}\,  c
\,\frac{\left|h_T\right|^n}{|h|^n} +B\big(|h|\big)
\left|h_N\right|\, \left|h_T\right|^{n-1}
\end{eqnarray}
Here  $$B \big(|h|\big)= n \rho(|h|)   \left(\eta^2(|h|) +n-1 \right)^\frac{n-2}{2}$$ comes from (\ref{eqb}). An important fact about
$B\big(|h|\big)$ is that we have  equality at \eqref{davidi1}  if
$\left|h_N \right|= \eta (|h|)\, \left|h_T \right|$. This hold true for the radial $(\rho, n)$-harmonic map at (\ref{vienna3}), by the
definition of the constant $c$. Let us integrate \eqref{davidi1}  over
the annulus ${\mathbb A}$. For the last term  we apply the
 lower bound at (\ref{E184}). To estimate the first term in
the right hand side of \eqref{davidi1}  we use H\"older's inequality and we have
$$
\left(\int_{\mathbb A} \frac{\left|h_T\right|^{n-1}}{|x|\, |h|^{n-1}}\right)^\frac{n}{n-1}  \le \int_{\mathbb A} \frac{\left|h_T\right|^n}{|h|^n} \left(\int_{\mathbb A} \frac{dx}{|x|^n}\right)^\frac{1}{n-1},$$
and then use ~\eqref{vienna12}. Thus we have
\begin{eqnarray}\label{vienna13}
\int_{\mathbb A} \rho(h)\norm Dh \norm^n &\ge & (n-1)^\frac{n}{2} c \left(\int_{\mathbb A} \frac{\left|h_T\right|^{n-1}}{|x|\, |h|^{n-1}}\right)^\frac{n}{n-1}  \left(\int_{\mathbb A} \frac{dx}{|x|^n}\right)^\frac{-1}{n-1}  \nonumber \\ &\; & +\,  \omega_{n-1} \int_{r_\ast}^{R_\ast} \tau^{n-1} B(\tau)  \, d \tau   \\
&\ge &    (n-1)^\frac{n}{2} \, c \, \mbox{Mod}\, {\mathbb A}\,
+ \, \omega_{n-1} \int_{r_\ast}^{R_\ast} \tau^{n-1} B(\tau) \, d
\tau \nonumber
\end{eqnarray}
Finally, observe that we have equalities in all estimates for the
radial stretchings. Thus
\begin{equation}\label{E273}
\int_{\mathbb A} \rho(h)\norm Dh \norm^n \ge \int_{\mathbb A} \rho(h_c) \norm
Dh_c \norm^n
\end{equation}
as stated.
\\
{$\bullet $ \bf The case $c_\diamond\le c\le 0$}. Then ${\mathbb A}^{\! \ast}$ is
  thinner than ${\mathbb A}$. Let $H=H_c$. Then
$$\mathcal{L}[H]=c\le 0.$$ Thus  $$c_1\equiv -c = (n-1)^{\frac{2-n}{2}}\rho(H)\left[t^2(\dot H)^2+(n-1)H^2\right]^{\frac{n-2}{2}}(t^2(\dot H)^2-H^2)\ge 0.$$
or $$(n-1+\eta^2) ^{\frac{n-2}{2}}(\eta^2-1)=\frac{c_1(n-1)^{\frac{n-2}{2}}}{|H|^np(H)}$$
Now we consider the  general mapping $h$. There is exactly one solution $\sigma=\eta(|h(x)|)$ of the equation
$$(n-1+\eta^2) ^{\frac{n-2}{2}}(\eta^2-1)=\frac{c_1(n-1)^{\frac{n-2}{2}}}{|h(x)|^np(|h(x)|)}.$$ Since $1\le |h(x)|\le R_*$, we conclude that
$$a(\sigma):=\frac{(n-1+\sigma^2) ^{\frac{n-2}{2}}(\sigma^2-1)}{\sigma^n}\le 1.$$
From Lemma~\ref{le7} we obtain
\begin{equation}\label{ps}\begin{split}\rho(|h|) (\|Dh\|^n)&= \rho(|h|) (|h_N|^2+(n-1)|h_T|^2)^{n/2}
\\&\ge \rho(|h|) \left(a(\sigma)|h_N|^n+ b(\sigma) |h_N|||h_T|^{n-1}\right)
\\&= c_1(n-1)^{\frac{n-2}{2}}\left[\frac{|h_N|}{|h|\eta(|h|)}\right]^n+\Phi(|h|)  |h_N|||h_T|^{n-1}\end{split}\end{equation} where
$$\Phi(|h|)=\rho(|h|) b(|\eta(|h|))=\rho(|h|) \frac{n\left(\eta^2(|h|) + n-1\right)^\frac{n-2}{2}}{\eta(|h|)}.$$
According to Lemma \ref{le7}, equality holds at a
given point $x$ if and only if $\left| h_N (x)\right|= \eta
\big(|h(x)|\big)\, \left|h_T(x)\right|$. In particular, it holds
almost everywhere for $h=h_c (x)$, because $\left|(h_c)_N
\right|=\eta_H \left|(h_c)_T \right|$. We now integrate over the
annulus ${\mathbb A}$. The last term at (\ref{ps}) is
estimated by using  (\ref{E184}),
\begin{eqnarray}\label{E300}
\int_{\mathbb A}  \Phi \big(|h|\big)  \left|h_N\right|\,
\left|h_T\right|^{n-1} &\ge &  \omega_{n-1}
\int_{r_\ast}^{R_\ast}
\tau^{n-1} \Phi(\tau)\, d\tau \nonumber \\
&=& \int_{{\mathbb A}} \Phi \big(|h_c |\big)\, \left|(h_c)_N
\right|  \left|(h_c)_T \right|^{n-1}
\end{eqnarray}
To estimate the first term in the right hand side of \eqref{ps} we make use of the
 identities
\begin{equation}
\frac{ \left|(h_c)_N\right|}{ \left|h_c \right|\, \eta
\big(|h_c |\big)} = \frac{\dot{H}}{H\, \eta_{_H}} =
\frac{1}{|x|}.
\end{equation}
Having in mind the simple inequality $\left|h\right|_N\le  \left|h_N\right|$, by using  H\"older's inequality we obtain
\begin{equation}\label{vienna1}
\left(  \int_{\mathbb A}
\frac{ \left|h\right|_N\, dx }{ \left|h\right|\, \eta
\big(|h|\big)\, |x|^{n-1}} \right)^n \le \int_{\mathbb A} \left[\frac{ \left|h_N\right|}{ \left|h\right|\,
\eta \big(|h|\big)}\right]^n \left(\int_{\mathbb A}
\frac{dx}{|x|^n}\right)^{n-1}.
\end{equation}
Further, as in the proof of \cite[Proposition~12.1]{memoirs}, we obtain
\begin{equation}
  \int_{\mathbb A} \frac{ \left|h\right|_N\, dx }{ \left|h\right|\, \eta \big(|h|\big)\, |x|^{n-1}} = \int_{\mathbb A}
  \frac{\left|h_c\right|_N\, dx}{\left|h_c\right|\, \eta \big(|h_c |\big)\, |x|^{n-1}}= \int_{\mathbb A} \frac{\dot{H}\big(|x|\big)\, dx }{H \, \eta_{_H} \, |x|^{n-1}} =  \int_{\mathbb A}
  \frac{dx}{|x|^n} \nonumber
\end{equation}
Hence
\begin{equation}\label{vienna2}
\int_{\mathbb A} \left[\frac{ \left|h_N\right|}{ \left|h\right|\,
\eta \big(|h|\big)}\right]^n \ge  \int_{\mathbb A}
\frac{dx}{|x|^n} =\mbox{Mod}\, {\mathbb A}.
\end{equation}
Thus
\begin{equation}
 \int_{\mathbb A} \rho(h) \, \norm Dh \norm^n \ge  c_1(n-1)^{\frac{n-2}{2}} \, \mbox{Mod}\, {\mathbb A} + \omega_{n-1}  \int_{r_\ast}^{R_\ast}
 \tau^{n-1} \Phi(\tau)\, d\tau,
\end{equation}
with equality attained for $h_c$, as stated. This finishes the proof of the fact that if the condition \eqref{Kalajnitsche1} is satisfied, then  we have the sharp inequality \eqref{sharpineq}. In order to prove the opposite statement, assume that $R^*>H_{c^\diamond}(R)$. Then by Theorem~\ref{newkalajrad} there is $c=c(R,R^*)<c^\diamond$ and a diffeomorphism $H=H_c:[1,R]\to [1,R^*]$, so that $h(x)=H(|x|)\frac{x}{|x|}$ is a $(\rho,n)$-harmonic diffeomorphism between $\mathbb{A}$ and $\mathbb{A}^\ast$.

This finishes the proof of Theorem~\ref{newkalaj}~a), up to the uniqueness part. The  uniqueness part follows by repetition the approach of the similar statement from \cite{memoirs}, and we will not write the details here. It is important to emphasis that in some key  places where we used the sharp inequalities, the equality statement is attained if and only if $$J_h(x)=\left|h_T\right|^{n-1}\left| h_{N} \right|$$ and so the matrix \begin{equation*}
{\bold C}(x,h) \bydef D^\ast h\cdot  Dh = \left[ \begin{array}{cccc}  \left| h_{N} \right|^2 & 0 & \cdots & 0  \\
 0 & \left| h_{T} \right|^2 &   \cdots & 0\\
& &\ddots &  \\
0 & 0 & \cdots & \left| h_{T} \right|^2
\end{array}\right]
\end{equation*}
enters to the stage, in order to prove that $h$ is radial.
\\
$\clubsuit$ {\bf Proof of b)}. Let ${\mathscr R}={\mathscr R} ({\mathbb A}\, , \, {\mathbb A}^\ast)$ be the class of orientation preserving radial $(\rho,n)$-harmonic diffeomorphisms mapping the inner boundary onto itself and let $\mathscr{D}=\mathscr{D}(R,R_\ast)$ be the class of orientation preserving $C^2$ diffeomorphisms of $[1,R]$ onto $[1,R_\ast]$.
Now, we find  the infimum in the left hand
side of \eqref{sharpineq} for $n>3$ and obtain
\begin{equation}\label{iwan1}\begin{split}
\inf_{h\in {\mathscr R}
}\int_{\mathbb A}\rho(h)\norm Dh\norm^n  &=\omega_{n-1} \inf_{H\in \mathscr{D}}\int_1^R \rho(H) \left[
\dot{H}^2 +
(n-1)t^{-2}H^2\right]^\frac{n}{2}t^{n-1}{dt}  .\end{split}
\end{equation}
Here  $$L(t, H,\dot H)\bydef \rho(H) \left[
\dot{H}^2 +
(n-1)t^{-2}H^2\right]^\frac{n}{2}t^{n-1}$$ is strictly convex in $K=\dot H$ and coercive and thus the minimum is attained for a smooth function $H_\circ$ satisfying Euler-Lagrange equation and boundary conditions $H_\circ(1)=1$ and $H_\circ(R)=R_\ast$. Then $H_\circ=H_c$.
In order to prove this fact notice that, in view of \eqref{negative} and \eqref{R} we obtain $R_\ast > R$. Thus \[\begin{split}\int_1^R \frac{|\dot H_\circ|}{H_\circ} dt&=\int_1^R\abs{d\log H_\circ(t)} dt\\&\ge \abs{\int_1^R{d\log H_\circ(t)} dt}
 \\&=\log R_\ast > \log R=\int_1^R\frac{1}{t}dt.\end{split}\] By \eqref{banane}  the expression $$\frac{|\dot H_\circ|}{H_\circ} -\frac{1}{t}$$ has a constant sign, and thus $\mathcal{L}[H_\circ]=c_1< 0$.

 So by \eqref{banane} we infer that $H'_\circ(t)>0$, and thus $H_\circ$ is an increasing diffeomirphism. But then it coincides with $H_c$, because of uniqueness of the solution under this constraint.
 We obtain that
 \begin{equation}\label{iwan}\begin{split}
\inf_{h\in {\mathscr R}
}\int_{\mathbb A}\rho(h)\norm Dh\norm^n  &=\omega_{n-1} \int_1^R \rho(H_c) \left[t^2
\dot{H_c}^2 +
(n-1){H_c}^2\right]^\frac{n}{2}\frac{dt}{t}  \\
&= \omega_{n-1}  \int_1^R \rho[H_c(t)] \left[H_c(t)\right]^n \left[\eta_{_{H_c}}^2(t) +
n-1\right]^\frac{n}{2}\frac{dt}{t}.\end{split}
\end{equation}

Let $\Phi^\lambda: S^{n-1}\to S^{n-1}$ be the so called spherical homothety constructed in \cite{memoirs}, where $\lambda> 0$ is a real parameter, so that $\Phi^1=\mathbf{Id}$. More precisely, if $(\theta, \varphi_1,\dots, \varphi_{n-2})$ are spherical coordinates of $x$, then $(\varphi(\theta), \varphi_1,\dots, \varphi_{n-2})$ are spherical coordinates of  $\Phi^\lambda(x)$, where $\varphi(\theta)=2\tan^{-1}(\lambda\tan\frac{\theta}{2})$. Then $\varphi$ is a diffeomorphism of $[0,\pi]$ onto itself.  Furthermore $\Phi^\lambda$ is a conformal self-mapping of the unit sphere. Thus if $\zeta=\mathcal{S}(\theta, \varphi_1,\dots, \varphi_{n-2})$ are spherical coordinates, and $\Phi^\lambda(\zeta)=\mathcal{S}(\varphi(\theta), \varphi_1,\dots, \varphi_{n-2})$, by using conformality of $ \Phi^\lambda$ and the formula $$\Phi^\lambda(\mathcal{S}(\theta, \varphi_1,\dots, \varphi_{n-2}))=\mathcal{S}(\varphi(\theta), \varphi_1,\dots, \varphi_{n-2}),$$
we obtain that the the ratio between Gram determinants of $$D\mathcal{S}(\varphi(\theta), \varphi_1,\dots, \varphi_{n-2})$$ and of  $$D\mathcal{S}(\theta, \varphi_1,\dots, \varphi_{n-2})$$ is equal to $\varphi'(\theta)^{n-1}$. Thus, having in mind the conformality of $\Phi^\lambda$ we define
  \begin{equation}\label{mye}|D\Phi^\lambda(\zeta)|={\varphi'(\theta)}=\frac{\sin\varphi(\theta)}{\sin \theta}=\frac{2 \lambda}{1+\lambda^2+(1-\lambda^2)\cos \theta},\end{equation} where $\theta\in[0,\pi]$ is the meridian of $\zeta$.

Notice that $\varphi$ is the only diffeomorphism that produces a conformal mapping on $\mathbf{S}^{n-1}$. Indeed it is only solution of the differential equation with respect to $\varphi$ in \eqref{mye}.

By \cite[Eq.~14.50]{memoirs} we have  \begin{equation}\label{E349}
\phi(\lambda):=\frac{1}{\omega_{n-1}}\int_{{\mathbf{S}}^{n-1}}\left[\sigma^2 + (n-1)  |{D\Phi^\lambda}|^2 \right]^\frac{n}{2} < \left[ \sigma^2 + n-1 \right]^\frac{n}{2}
\end{equation} for every parameter $1< \lambda\sqrt{\frac{n-3}{n-1}}\sigma,$ where $\sigma> \sqrt{\frac{n-1}{n-3}}$.

This mean that $\lambda=1$ is a local maximum of $\phi$. We prove here more, $\lambda=1$ is local maximum of $\phi$ if and only if $\sigma>\sqrt{\frac{n-1}{n-3}}$.

Then, by direct computation, in view of \eqref{mye}  we find that $$\phi'(1)=0$$ and $$\phi''(1)=\frac{2 \left(-1-\sigma^2 (-3+n)+n\right) \sqrt{\pi } \Gamma\left[\frac{1+n}{2}\right]}{ \Gamma\left[\frac{n}{2}\right]}.$$ So $\phi''(1)<0$ if and only if $\sigma> \sqrt{\frac{n-1}{n-3}}$.

Then we test the infimum in  the right hand side of \eqref{sharpineq}
with the  mapping
\begin{equation}
h_\lambda (x)= H\big( |x|\big)\, \Phi^\lambda
\left(\frac{x}{|x|}\right)
\end{equation}
where, as in the previous case, $\Phi^\lambda : \mathbf{S}^{n-1}
\rightarrow \mathbf{S}^{n-1}$ is the spherical homothety and $H=H_c$. An important facts concerning $\Phi^\lambda$, which follows from \eqref{mye}, is the following
$$\int_{\mathbf{S}^n} {|D\Phi^\lambda|^{n-1}}=\omega_{n-1}$$
 From the equation  $$(\eta^2(t)-1)\left(1+\frac{\eta^2(t)}{n-1}\right)^{(n-2)/2}=-v_c(H(t))=\frac{-c}{\rho[H(t)]H(t)^n}$$  in view of \eqref{eta} we infer that

$$(\eta^2(t)-1)\left(1+\frac{\eta^2(t)}{n-1}\right)^{(n-2)/2}\ge \frac{-c}{\rho(R_*)R_*^n}.$$ From \eqref{cr} we obtain $$(\eta^2(t)-1)\left(1+\frac{\eta^2(t)}{n-1}\right)^{(n-2)/2}> \left(\left(\sqrt{\frac{n-1}{n-3}}\right)^2-1\right)\left(1+\frac{\left(\sqrt{\frac{n-1}{n-3}}\right)^2}{n-1}\right)^{(n-2)/2},$$ and thus \begin{equation}\label{newk}\eta=\eta_H(t)> \sqrt{\frac{n-1}{n-3}}.\end{equation}
From \eqref{iwan} and \eqref{E349} we find
that
\begin{eqnarray}
\inf_{h\in {\mathscr P} ({\mathbb A}\, , \, {\mathbb A}^\ast)
}\int_{\mathbb A}\rho(h)\norm Dh\norm^n &\le &  \int_{\mathbb A}\rho(h_\lambda)\norm
Dh_\lambda \norm^n
\nonumber \\
&=&   \int_r^R \rho[H(t)]\left[H(t)\right]^n\int_{\mathbf{S}^n}
\left[\eta_{_H}^2(t) + (n-1)| D\Phi^\lambda |^2
\right]^\frac{n}{2}\frac{dt}{t} \nonumber \\
&<& \omega_{n-1} \int_r^R \rho[H(t)]\left[H(t)\right]^n \left[\eta_{_H}^2(t) +
n-1\right]^\frac{n}{2}\frac{dt}{t} \nonumber  \\
&=& \inf_{h\in {\mathcal R} ({\mathbb A}\, , \, {\mathbb A}^\ast)
}\int_{\mathbb A}\rho(h)\norm Dh\norm^n \nonumber.
\end{eqnarray}

Here we have chosen
$\lambda > 1$ sufficiently close to 1.

\end{document}